\documentclass{amsart}
\usepackage[utf8]{inputenc}
\usepackage{amsmath}
\usepackage{amssymb}
\usepackage{caption}
\usepackage{amsthm}

\usepackage[dvipsnames]{xcolor}
\usepackage[backref=false, colorlinks, linkcolor=blue, citecolor=Blue]{hyperref}\usepackage{cleveref}
\usepackage{nicefrac}
\usepackage[inline]{enumitem}
\usepackage{enumitem}

\usepackage{amssymb}
\usepackage{mathrsfs}

\usepackage{array,float}

\usepackage{tikz}
\usetikzlibrary{shapes,arrows}
\usetikzlibrary{fit,positioning}
\usetikzlibrary{patterns,decorations.pathreplacing}
\usepackage{xkeyval}
\usepackage{moreverb}

\usepackage{epic}
\usepgfmodule{shapes,plot,decorations}
\usepackage{booktabs} 
\usepackage[normalem]{ulem} 

\newcommand{\HH}{\mathbb{H}}

\usepackage[warn]{mathtext}

\newcommand{\Isom}{\mathrm{Isom}}

\newtheorem{theorem}{Theorem}[section]
\newtheorem{corollary}[theorem]{Corollary}

\newtheorem{proposition}[theorem]{Proposition}

\theoremstyle{remark}

\theoremstyle{remark}

\renewcommand{\qed}{\hfill$\scriptstyle\blacksquare$}

\title{Arithmetic trialitarian hyperbolic lattices are not LERF}
\author{Nikolay Bogachev}
\address{Department of Computer and Mathematical Sciences, University of Toronto Scarborough, 1095 Military Trail, Toronto, ON M1C 1A3, Canada}
\address{Institute for Information Transmission Problems, Moscow, Russia}
\email{n.bogachev@utoronto.ca}

\author{Leone Slavich}
\address{Dipartimento di Matematica, Universit\`a di Pavia, Via Ferrata 5, 27100 Pavia, Italy}
\email{leone.slavich@gmail.com}

\author{Hongbin Sun}
\address{Department of Mathematics, Rutgers University - New Brunswick, Hill Center, Busch Campus, Piscataway, NJ 08854, USA}
\email{hongbin.sun@rutgers.edu}

\begin{document}

\begin{abstract}
A group is LERF (locally extended residually finite) if all its finitely generated subgroups are separable. We prove that the trialitarian arithmetic lattices in $\mathbf{PSO}_{7,1}(\mathbb{R})$ are not LERF. This result, together with previous work by the third author, implies that no arithmetic lattice in $\mathbf{PO}_{n,1}(\mathbb{R})$, $n>3$, is LERF.
\end{abstract}

\maketitle

\section{Introduction}

Let $\HH^n$ denote the real hyperbolic
$n$-space and $\Isom(\HH^n)$ be its isometry group.
Here and below a \textit{hyperbolic lattice} is a discrete subgroup of $\Isom(\HH^n)$ admitting a finite-volume fundamental polyhedron in $\mathbb{H}^n$.

The main goal of this paper is to study separability properties of the exceptional family of arithmetic trialitarian lattices in $\mathrm{Isom}(\HH^7)$. Let $G$ be a group and $H<G$ a subgroup. We say that $H$ is \emph{separable} in $G$ if for every $g \in G \setminus H$ there exists a finite-index subgroup $K < G$ such that $g \not\in K$ but $H < K$. If all finitely generated subgroups of $G$ are separable, then $G$ is said to be {\em LERF} (i.e. {\em locally extended residually finite}). If $G$ is a subgroup of $\Isom(\HH^n)$, then $G$ is said to be {\em GFERF} (i.e. {\em geometrically finite extended residually finite}) if all geometrically finite subgroups of $G$ are separable.

Separability properties have profound implications in both geometric group theory and geometric topology. Finitely generated free groups are LERF, as proved by Hall \cite{Ha49}. By a theorem of Scott \cite{Sc78}, all finitely generated Fuchsian groups are LERF as well. That the same is true for finitely generated Kleinian groups is a deep result involving much of what is known about hyperbolic $3$-manifold groups; see Agol \cite{Agol13} and Wise \cite{Wise21} (for a detailed history of this result, see \cite{AFW}). On the other hand, the third author proved that almost all arithmetic hyperbolic lattices in higher dimensions fail to be LERF \cite{Sun19a, Sun19b}, though those of simplest type are nevertheless GFERF by work of Bergeron, Haglund, and Wise \cite{BHW11}.

The special $7$-dimensional arithmetic hyperbolic lattices that we consider are known as {\em trialitarian} lattices, and are characterised by the property of being commensurable with the group of integer points of a {\itshape trialitarian} $k$-form $\mathbf{G}$ of the real group $\mathbf{PSO}_{7,1}$, where $k$ is a totally real algebraic number field. This means that the natural action of the absolute Galois group $\mathrm{Gal}(\overline{k}/k)$ on $\mathbf{G}$ induces an order-$3$ automorphism of the Dynkin diagram of the $D_4$ root system \cite{Tits}. Given that $\mathbf{PSO}_{7,1}$ is an outer form of the complex group $\mathbf{PSO}_{8}$, it follows that this action is the full action by automorphisms of the diagram, i.e.\ the group $\mathbf{G}$ has Tits symbol $^{6}D_4$.

Very little is known about the properties of trialitarian hyperbolic lattices. The most notable result is due to Bergeron and Clozel \cite{BC17}. They prove that if $\Gamma$ is a congruence subgroup of a trialitarian group $\mathbf{G}$ as above, then $b_1(\Gamma)=0$. This is in sharp constrast with the case of all other arithmetic lattices in $\mathrm{Isom}(\HH^n)$, $n>3$, for which the existence of congruence covers with positive first Betti number is known (see \cite{Schwermer} for a review of the subject).
 
The main result of this paper is the following:

\begin{theorem}\label{main}
    Arithmetic trialitarian lattices in $\Isom(\HH^7)$ are not LERF.
\end{theorem}

By combining Theorem \ref{main} with \cite[Theorem 1.2]{Sun19a} and \cite[Theorems 1.1, 1.2]{Sun19b}, we obtain the following corollary:

\begin{corollary}
No arithmetic lattice in $\Isom(\HH^n)$ for $n>3$ is LERF.
\end{corollary}

It was demonstrated by the third author of this paper that many nonarithmetic lattices in $\Isom(\HH^n)$, $n>3$, also fail to be LERF; see \cite[Theorems 5.2, 5.3]{Sun19b}. We conjecture that no lattice in $\Isom(\HH^n)$ with $n>3$ is LERF.

\subsection*{Acknowledgements} 

    We are grateful to Misha Belolipetsky and Sami Douba for fruitful discussions, and for their comments on this note. We would also like to express our gratitude to David Futer for organizing the conference ``Groups around $3$-manifolds'', and to Daniel Wise for organizing the broader thematic program on Geometric Group Theory at the CRM in Montreal, where this work was initiated. The work of H.S. was partially supported by the Simons Collaboration Grant 615229. The work of L.S. was supported by the PRIN project ``Geometry and topology of manifolds'' project no.~F53D23002800001 and by the INdAM Institute.

\section{Trialitarian lattices and manifolds}\label{sec:trialitarian-lattices}
In this section, we recall the definition and basic properties of trialitarian hyperbolic manifolds.

Consider the hyperboloid model for the $7$-dimensional hyperbolic space:
$$\HH^7=\{(x_0,x_1,\dots,x_7)\in \mathbb{R}^7 \mid x_0>0,\; f(x_0,x_1,\dots,x_7)=-1\},$$
where $f(x_0,x_1,\dots,x_7)=-x_0^2+x_1^2+\dots+x_7^2$ denotes the (standard) form of signature $(7,1)$ and the Riemannian metric on $\HH^7$ is induced by the restriction of $f$ to the tangent bundle. With this model, the orientation-preserving isometry group of $\HH^7$ is easily seen to be isomorphic to the real Lie group $\mathbf{PSO}_{7,1}(\mathbb{R})$.

A \emph{trialitarian lattice} is any subgroup of $\mathbf{PSO}_{7,1}(\mathbb{R})$ which can be constructed as follows.
Consider a totally real algebraic number field $k$ and $\mathbf{G}$ a $k$-form of the real algebraic group $\mathbf{PSO}_{7,1}$ with Tits symbol ${^6}D_{4}$. We furthermore require that $\mathbf{G}$ is \emph{admissible}, meaning that all the groups $\mathbf{G}^{\sigma}$ obtained by applying a non-identity Galois embedding $\sigma:k\rightarrow \mathbb{R}$ to the polynomial equations that define $\mathbf{G}$ have a compact group $\mathbf{G}^{\sigma}(\mathbb{R})$ of real points isomorphic to $\mathbf{PSO}_{8}(\mathbb{R})$. Notice that algebraic groups with the above properties do indeed exist \cite{BC13}.

It follows from a celebrated theorem of Borel and Harish-Chandra \cite{BHC62} that the group $\mathbf{G}(\mathcal{O}_k)$ of integer points of $\mathbf{G}$ is an arithmetic lattice in $\mathbf{PSO}_{7,1}(\mathbb{R})$, i.e. it acts on hyperbolic space $\mathbb{H}^7$ properly discontinuously by isometries and with finite covolume.

A trialitarian lattice is any arithmetic lattice $\Gamma <\mathbf{PSO}_{7,1}(\mathbb{R})$ which is commensurable in the wide sense with a group of the form $\mathbf{G}(\mathcal{O}_k)$ as above. All trialitarian lattices are cocompact, meaning that their action on $\mathbb{H}^7$ admits a compact fundamental domain. A \emph{trialitarian hyperbolic orbifold} is an orbifold of the form $M=\mathbb{H}^7/\Gamma$ with $\Gamma$ a trialitarian lattice. It is a (closed) manifold if $\Gamma$ is also torsion-free. Notice that this latter condition can always be achieved up to passing to a finite-index subgroup $\Gamma' < \Gamma$ due to Selberg's lemma.

Let us consider a trialitarian lattice $\Gamma<\mathbf{PSO}_{7,1}(\mathbb{R})$. We may assume up to conjugation that $\Gamma$ is commensurable with $\mathbf{G}(\mathcal{O}_k)$. Since $\mathbf{G}$ is an adjoint algebraic  group, it follows from work of Vinberg \cite{Vin71} that $\Gamma$ is contained in the group of rational points $\mathbf{G}(k)$. Moreover, work of Borel \cite{Borel-adjoint} shows that the commensurator of $\Gamma$ in $\mathbf{G}(\mathbb{R})\cong \mathbf{PSO}_{7,1}(\mathbb{R})$ is given precisely by the group $\mathbf{G}(k)$.

\section{Geodesic subspaces}
In this section we prove that any hyperbolic trialitarian manifold $M=\mathbb{H}^7/\Gamma$ contains three immersed totally geodesic $3$-manifolds that intersect along a geodesic.

As in Section \ref{sec:trialitarian-lattices}, we consider the hyperboloid model for the $7$-dimensional hyperbolic space $\HH^7$ and identify the orientation-preserving isometry group of $\HH^7$ with $\mathbf{PSO}_{7,1}(\mathbb{R})$.

Consider the following three orientation-preserving isometric involutions of $\HH^7$:
\begin{align*}
\theta_1(x_0,x_1,x_2,x_3,x_4,x_5,x_6,x_7)=(x_0,x_1,x_2,x_3,-x_4,-x_5,-x_6,-x_7);\\
\theta_2(x_0,x_1,x_2,x_3,x_4,x_5,x_6,x_7)=(x_0,x_1,-x_2,-x_3,x_4,x_5,-x_6,-x_7);\\
\theta_3(x_0,x_1,x_2,x_3,x_4,x_5,x_6,x_7)=(x_0,x_1,-x_2,-x_3,-x_4,-x_5,x_6,x_7).
\end{align*}

It is clear that $\theta_3=\theta_1 \circ \theta_2$ and that, moreover, $\theta_1$ and $\theta_2$ commute. As such, the involutions $\theta_1$ and $\theta_2$ generate a subgroup $K$ of $\mathbf{PSO}_{7,1}(\mathbb{R})$ isomorphic to the Klein group $\mathbb{Z}/2\mathbb{Z} \times \mathbb{Z}/2\mathbb{Z}$.

The main result of this section is the following:
\begin{theorem}\label{prop:Klein-group-commensurator}
Let $\Gamma<\mathbf{PSO}_{7,1}(\mathbb{R})$ be a trialitarian lattice. The commensurator of $\Gamma$ in $\mathbf{PSO}_{7,1}(\mathbb{R})$ contains a subgroup which is conjugate to the group $K$.
\end{theorem}
We notice that an analogous statement in the setting of arithmetic lattices in $\mathbf{PSL}_{2}(\mathbb{C})$ was proven by Lackenby, Long and Reid in \cite[Theorem 1.2]{LaLoRe}, and later played a key role in Lackenby's proof of the surface subgroup conjecture for arithmetic Kleinian groups \cite{Lackenby}.

Before proving the Theorem above, we review some key properties of trialitarian algebraic groups. We denote by $\Sigma$ the $D_4$ root system, realised as the set of $24$ vector of $\mathbb{R}^4$ obtained by permuting the entries of $(\pm1,\pm1,0,0)$, and by $\Delta=\{\alpha_0,\alpha_1,\alpha_2,\alpha_3\}$ the following set of simple roots:
\begin{equation*}
\alpha_0=(0,1,-1,0),\, \alpha_1=(1,-1,0,0),\, \alpha_2=(0,0,1,-1),\, \alpha_3=(0,0,1,1),
\end{equation*} with $\alpha_0$ corresponding to the ``central'' root in the Dynkin diagram (see Figure~\ref{fig:D4_dynkin}). The automorphism group of the Dynkin diagram for $\Delta$ is isomorphic to the symmetric group $\mathfrak{S}_3$, which acts by fixing $\alpha_0$ and permuting the other $3$ roots.

\begin{figure}[h]
    \centering
    \includegraphics[width=3.2cm]{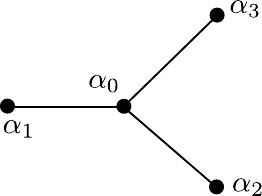}
    \caption{The Dynkin diagram of the $D_4$ root system, with roots labeled.}
   \label{fig:D4_dynkin}
\end{figure}

We recall the following result from \cite[proof of Proposition 3.17]{BBKS}:
\begin{proposition}\label{prop:torusaction}
Let $k$ be a totally real algebraic number field and $\mathbf{G}$ a $k$-form of the real algebraic group $\mathbf{PSO}_{7,1}$ with Tits symbol ${^6}D_{4}$. There exists a maximal $k$-torus $\mathbf{T}<\mathbf{G}$ with the following properties:
\begin{enumerate}
    \item[\textnormal{(i)}] The action of the absolute Galois group $\mathcal{G}$ of $k$ on the root system $\Sigma$ of $\mathbf{G}$ relative to $\mathbf{T}$ is isomorphic to $\mathfrak{S}_3 \times \mathbb{Z}/2\mathbb{Z}$, with the $\mathfrak{S}_3$ factor acting on the set $\Delta$ of simple roots by automorphisms of its Dynkin diagram, and $\mathbb{Z}/2\mathbb{Z}$ acting via the antipodal map.
    \item[\textnormal{(ii)}] The action of complex conjugation $\sigma \in \mathcal{G}$ is given by 
    \begin{equation}\label{eq:conjugation-action}(x,y,z,w) \xrightarrow{\sigma} (-x,-y,-z,w).\end{equation}
\end{enumerate}
\end{proposition}

We begin by proving the following Proposition, which essentially states that $\mathbf{T}$ contains a maximal $\mathbb{R}$-split torus:
\begin{proposition}\label{prop:maximal-split-torus}
The group $\mathbf{T}(\mathbb{R})$ of real points of the torus $\mathbf{T}$ is isogenous to the real Lie group $\mathbf{SO}_{1,1}(\mathbb{R}) \times \mathbf{SO}_2(\mathbb{R})^3$.
\end{proposition}

\begin{proof}
The root system $\Sigma$ of $\mathbf{G}$ relative to $\mathbf{T}$ is identified in the obvious way with a subset of the character group $\mathbf{X}^*(\mathbf{T}) \cong \mathbb{Z}^4$ of the torus $\mathbf{T}$. The action of complex conjugation \eqref{eq:conjugation-action} is an involution with eigenvalues $1$ and $-1$ with multiplicities $1$ and $3$ respectively. As such there is a one-dimensional subspace of $\mathbf{X}^*(\mathbf{T})$ of $\mathbb{R}$-defined characters, implying that the maximal $\mathbb{R}$-split torus of $\mathbf{T}$ is one-dimensional. 
\end{proof}

\subsubsection*{Proof of Theorem~\ref{prop:Klein-group-commensurator}}
As explained in Section \ref{sec:trialitarian-lattices}, the commensurator of a trialitarian lattice $\Gamma$ can be identified with the group $\mathbf{G}(k)$ of rational points of the trialitarian group $\mathbf{G}$. As such we need to find a pair of commuting involutions in $\mathbf{G}(k)$ and check that they generate a group of hyperbolic isometries in $\mathbf{G}(\mathbb{R})\cong \mathbf{PSO}_{7,1}(\mathbb{R})$ which is conjugate to $K$. 

Recall that $\mathbf{G}$ is a connected, adjoint algebraic $k$-group. As such it is isomorphic to the group of its inner automorphisms. Any order-$2$ element in $\mathbf{G}(k)$ thus corresponds to an involutory inner $k$-automorphism of $\mathbf{G}$.

In order to find the generators for the group $K$, we look at the rational points $\mathbf{T}(k) < \mathbf{G}(k)$ of the maximal torus $\mathbf{T}$. As explained in \cite[Section 3.4.1]{BBKS}, conjugation by elements of $\mathbf{T}$ induces the trivial action on the root system $\Sigma$ of $\mathbf{G}$ relative to $\mathbf{T}$. An order $2$ element $g$ of $\mathbf{T}$ is thus required to act as $\pm\mathrm{id}$ on each root space $L_{\alpha},\;\alpha \in \Sigma$. 

Moreover the action of conjugation by $g$ on each root space is uniquely determined by the action on the root spaces for the set $\Delta$ of simple roots \cite[\S 14.2]{Hum}. As such any partition of the set $\Delta=\Delta_+ \cup \Delta_-$ into two subsets (with non-empty $\Delta_-$) determines an order-two element $g$ in $\mathbf{T}$ via the rule that $\theta=\mathrm{Inn}(g)$ acts as $+\mathrm{id}$ and $-\mathrm{id}$ on the root spaces for the roots in $\Delta_+$ and $\Delta_-$, respectively. The action on the other root spaces is determined as follows: every root $\alpha \in \Sigma$ is expressed as a linear combination of the roots $\alpha \in \Delta$ with integer coefficients:

$$\sum_{\alpha \in \Delta_+} n_{\alpha}\cdot \alpha + \sum_{\beta \in \Delta_-} m_{\beta}\cdot \beta.$$
The roots for which $\sum_{\beta \in \Delta_-}m_{\beta}$ is even (resp.\ odd) are called \emph{even} (resp.\ \emph{odd}). The set of even (resp.\ odd) roots is denoted by $\Sigma_+$ (resp.\ $\Sigma_-$). The automorphism $\theta$ acts as $+\mathrm{id}$ (resp.\ $-\mathrm{id}$) on the roots spaces for the roots of $\Sigma_+$ (resp.\ $\Sigma_-$). 

In order for $g$ to lie in the group $\mathbf{T}(k)$ of rational points, we need the inner automorphism $\mathrm{Inn}(g)$ to commute with the natural action of the absolute Galois group $\mathrm{Gal}(\overline{k}/k)$ of $k$. This is easily seen to translate to the requirement that this action preserves the sets $\Sigma_+$ and $\Sigma_-$.

The action of the absolute Galois group is completely described in Proposition~\ref{prop:torusaction}. The antipodal map maps every root to its opposite and therefore preserves both sets $\Sigma_+$ and $\Sigma_-$.
As such we need to choose the partition $\Delta=\Delta_+ \cup \Delta_-$ in such a way that both subsets are preserved by the action of the $\mathfrak{S}_3$ factor.
There are only $3$ possible choices:
\begin{align} 
\Delta_-=\{\alpha_0\}, &\; \Delta_+=\{\alpha_1,\alpha_2,\alpha_3\}; \label{eq:root-partition-1} \\ 
\Delta_+=\{\alpha_0\}, &\; \Delta_-=\{\alpha_1,\alpha_2,\alpha_3\}; \label{eq:root-partition-2} \\ 
\Delta_-=\Delta,\;\;\;\; &\; \Delta_+=\emptyset. \label{eq:root-partition-3}
\end{align}
Each choice corresponds to an involutory inner $k$-automorphism $\theta$ of $\mathbf{G}$, and thus to an order-two element $g \in \mathbf{G}(k)$ such that $\mathrm{Inn}(g)=\theta$. Let us denote by $\theta_1$, $\theta_2$ and $\theta_3$ the $k$-involutions of $\mathbf{G}$ described by \eqref{eq:root-partition-1}, \eqref{eq:root-partition-2} and \eqref{eq:root-partition-3} respectively, and by $g_1, g_2$ and $g_3$ the corresponding elements in $\mathbf{G}(k)$. It is clear that $\theta_3=\theta_1 \circ \theta_2$ and thus $g_3=g_1 \cdot g_2$. This implies that the group $\langle g_1,g_2\rangle<\mathbf{G}(k)$ is isomorphic to $\mathbb{Z}/2\mathbb{Z} \times \mathbb{Z}/2\mathbb{Z}$.

We now check that the fixed-point-set for the action on $\mathbb{H}^7$ of $g_i$, $i=1,2,3$, is $3$-dimensional. In order to do so, we compute the dimension of the centraliser of $g_i$ in $\mathbf{G}(\mathbb{R})\cong \mathbf{PSO}_{7,1}(\mathbb{R})$, i.e.\ the dimension of the fixed-point-set of $\theta_i$ in $\mathbf{PSO}_{7,1}(\mathbb{R})$.

Given that $\theta_i^2=\mathrm{id}$, it follows that the adjoint action of $\theta_i$ on the real Lie algebra $\mathfrak{so}_{7,1}$ is diagonalisable and has eigenvalues $1$ and $-1$ with multiplicities $m(1)$ and $m(-1)$ respectively. The multiplicities can be computed easily: $\theta_i$ acts as the identity on the $4$-dimensional algebra $\mathfrak{t}$ of $\mathbf{T}$ and on each one-dimensional root space $L_{\alpha}$ for $\alpha \in \Sigma_+$ an even root. It also acts as $-\mathrm{id}$ on each root space $L_{\alpha}$ for $\alpha \in \Sigma_-$ an odd root.

For all the possible choices (\eqref{eq:root-partition-1},~\eqref{eq:root-partition-2},~\eqref{eq:root-partition-3})  for $\Delta_+$ and $\Delta_-$ we obtain that $\Sigma_+$ has $8$ roots and $\Sigma_-$ has $16$ roots.
We list the roots in $\Sigma_+$ for each case:
\begin{align*}
\Sigma^1_+=\{(\pm1,\pm1,0,0)\}\cup \{(0,0,\pm1,\pm1)\},\\
\Sigma^2_+=\{(\pm1,0,0,\pm1)\}\cup \{(0,\pm1,\pm1,0)\},\\
\Sigma^3_+=\{(0,\pm1,0,\pm1)\}\cup \{(\pm1,0,\pm1,0)\}.
\end{align*}

It follows that we have $m(1)=12,\, m(-1)=16$. Since the resulting involution $\theta_i$ has order $2$, it can be represented by the conjugation by a matrix $M \in \mathbf{SO}_{7,1}(\mathbb{R})$ such that $M^2=\mathrm{id}$. Up to conjugacy in $\mathbf{SO}_{7,1}(\mathbb{R})$ we may assume that $M$ is diagonal with $\pm 1$ entries on the diagonal.
The only possibility such that the action on the Lie algebra has $m(1)=12$, $m(-1)=16$ is that $M$ has $4$ entries equal to $1$ and $4$ entries equal to $-1$. Such an $M$ corresponds to a reflection along a $4$ dimensional subspace $V$ in the quadratic space $(\mathbb{R}^8,f)$. Up to multiplication by $-\mathrm{Id}$, we may assume that $V$ intersects the hyperboloid $\mathbb{H}^7$ in a $3$-dimensional totally geodesic subspace, and we denote by $V_i$, $i=1,2,3$, the fixed-point-set of the reflection for each of the possible choices \eqref{eq:root-partition-1},~\eqref{eq:root-partition-2},~\eqref{eq:root-partition-3}.

Notice how the three possible sets of even roots $\Sigma^i_+$ form a partition of $\Sigma$ into $3$ sets of $8$ elements. 
Since any two such sets are disjoint,
it follows that the intersection of the centralisers of $\theta_i$ and $\theta_j$ (and, in fact, the centraliser of all three involutions) is given by the maximal torus $\mathbf{T}(\mathbb{R})$. In fact it is easy to see that there is an isomorphism
\begin{equation}\label{eq:max-R-torus}
\mathbf{T}(\mathbb{R}) \cong \frac{\mathbf{SO}_{1,1}(\mathbb{R}) \times \mathbf{SO}_2(\mathbb{R})^3}{\mathbb{Z}/2\mathbb{Z}},
    \end{equation}
with $\mathbb{Z}/2\mathbb{Z}$ acting ``diagonally'' as $-\mathrm{id}$ in each factor. The group $\mathbf{SO}_{1,1}(\mathbb{R}) \times \mathbf{SO}_2(\mathbb{R})^3$ is a maximal torus in $\mathbf{SO}_{7,1}(\mathbb{R})$ and its factors correspond to an orthogonal direct sum decomposition of the quadratic space $(\mathbb{R}^8,f)$:
\begin{equation}\label{eq:torus-decomposition}
(\mathbb{R}^8,f)=(W_0,f_0)\oplus(W_1,f_1)\oplus(W_2,f_2)\oplus(W_3,f_3)
\end{equation} where $\mathrm{dim}(W_i)=2$ and $f_0$ has signature $(1,1)$ while $f_1,f_2,f_3$ are positive definite. 

A direct computation shows that, for any possible choice of $\Delta_+$ and $\Delta_-$ as in \eqref{eq:root-partition-1},~\eqref{eq:root-partition-2},~\eqref{eq:root-partition-3}, there exists a reflection in a subspace of the form $W_0\oplus W_i$, $i=1,2,3$ that acts as $+\mathrm{id}$ (resp.\ $-\mathrm{id}$) on the root spaces for the roots of $\Delta_+$ (resp.\ $\Delta_-$).

This is sufficient to conclude that, up to an appropriate choice of the indices, $V_i=W_0 \oplus W_i$, and this implies that the group $\langle g_1,g_2 \rangle < \mathbf{G}(k)=\mathrm{Comm}(\Gamma)$ is conjugate to $K$ in $\mathbf{PSO}_{7,1}(\mathbb{R})$. \qed

\medskip
We conclude this section by using Theorem~\ref{prop:Klein-group-commensurator} to prove the following result:

\begin{proposition}\label{prop:amalgamation}
Let $M=\mathbb{H}^7/\Gamma$ be a trialitarian manifold. Then $M$ contains a closed, immersed geodesic $\gamma$ and $3$-dimensional closed, immersed, totally geodesic subspaces $S_1$, $S_2$ and $S_3$ with the property that, for any $i \neq j$, $\gamma\subset S_i \cap S_j$ and $S_i$ intesects $S_j$ orthogonally along $\gamma$.
\end{proposition}

We will make use of the following result \cite[Theorem 1.9]{BBKS}.
\begin{proposition}\label{prop:fc-subspace}
Let $\Gamma < \Isom(\HH^n)$ be a lattice and $F < \Isom(\HH^n)$ be a finite subgroup, such that $H = \mathrm{Fix}(F)$ is an $m$--dimensional plane in $\HH^n$, with $m\geq 2$. If $F < \mathrm{Comm}(\Gamma)$, then the stabiliser $\mathrm{Stab}_{\Gamma}(H)$ of $H$ in $\Gamma$ is a lattice acting on $H$.
\end{proposition}

\subsubsection*{Proof of Proposition~\ref{prop:amalgamation}}
We apply Proposition \ref{prop:fc-subspace} with $\Gamma$ a torsion-free trialitarian lattice and $F_1,\,F_2,\, F_3\cong \mathbb{Z}/2\mathbb{Z}$ the groups generated by the nontrivial elements in the group $K<\mathrm{Comm}(\Gamma)$ of Theorem~\ref{prop:Klein-group-commensurator}. We denote by $H_1$, $H_2$ and $H_3$ the fixed point sets in $\mathbb{H}^7$ of $F_1$, $F_2$ and $F_3$ respectively, which are totally geodesic copies of $\mathbb{H}^3$. We obtain that each $S_i=H_i/\mathrm{Stab}_{\Gamma}(H_i)$, $i=1,2,3$, is a closed, totally geodesic, immersed hyperbolic $3$-manifold in $M$.

Denote by $\widetilde{\gamma} \subset \mathbb{H}^7$ the geodesic at the intersection of $H_1$, $H_2$ and $H_3$.
The intersection of any two geodesic subspace $S_i \cap S_j$ contains the image $\gamma$ of $\widetilde{\gamma}$ under the covering map $\mathbb{H}^7 \rightarrow M= \mathbb{H}^7/\Gamma$. Since $\Gamma$ is a cocompact lattice, $\gamma$ is a closed, immersed geodesic in $M$. Since $H_i$ intersects $H_j$ orthogonally along $\widetilde{\gamma}$, any two geodesic subspaces $S_i$ and $S_j$ intersect orthogonally along $\gamma$. \qed

\section{Proof of Theorem~\ref{main}}

Given Proposition \ref{prop:amalgamation}, we are ready to prove the nonLERFness of trialitarian lattices in $\mathbf{PSO}_{7,1}(\mathbb{R})$.

\begin{proof}[Proof of Theorem~\ref{main}]
By Selberg's lemma, any trialitarian lattice $\Gamma$ contains a torsion-free finite-index subgroup $\Gamma'$. Then $M=\mathbb{H}^7/\Gamma'$ is a trialitarian manifold.

By Proposition \ref{prop:amalgamation}, there are two closed hyperbolic $3$-manifolds $S_1$ and $S_2$ that immerse into $M$ as totally geodesic subspaces, such that $S_1$ and $S_2$ intersect orthogonally along a closed geodesic $\gamma$. We only use two of the three totally geodesic submanifolds in Proposition \ref{prop:amalgamation}.

Let us give the closed geodesic $\gamma$ an orientation. This gives a closed oriented geodesic $c_i$ in each $S_i$, with $i=1,2$. Let $\mathbb{Z}\to \pi_1(S_i)$ be the group inclusion that maps $1$ to the element represented by $c_i$, and we denote by $\pi_1(S_1)*_{\mathbb{Z}}\pi_1(S_2)$ the amalgamation of $\pi_1(S_1)$ and $\pi_1(S_2)$ along $\mathbb{Z}$. By \cite[Lemma 7.1]{BHW11} there exists a finite-index subgroup $\Gamma''$ of $\Gamma'$ such that $\pi_1(S_1) < \Gamma''$, $S_2$ lifts to an immersed geodesic $3$-manifold $S_2'$ in $\mathbb{H}^7/\Gamma''$, and such that the amalgamation $\pi_1(S_1)*_{\mathbb{Z}}\pi_1(S_2')$ injects in $\Gamma''$ (and thus in $\Gamma$). We notice that Lemma 7.1 of \cite{BHW11} is stated only for arithmetic lattices of simplest type, but the proof applies without modifications to the case of trialitarian lattices thanks to the separability of the stabilisers of totally geodesic subspaces in hyperbolic lattices (see \cite[Lemma 1.8]{BHW11}).

By Theorem 1.4 of \cite{Sun19b}, $\pi_1(S_1)*_{\mathbb{Z}}\pi_1(S_2')$ is not LERF, since it is an amalgamation of two hyperbolic $3$-manifold groups along an infinite cyclic group. Since we have that $\pi_1(S_1)*_{\mathbb{Z}}\pi_1(S_2')$ is a subgroup of $\Gamma$ and the property of being LERF passes to subgroups, we see that $\Gamma$ is not LERF.
\end{proof}

\bibliography{biblio.bib}{}
\bibliographystyle{siam}

\end{document}